\documentclass{amsart}

  \usepackage[all]{xy}

  \usepackage{epsf,epsfig,amsfonts,graphicx,color}

\usepackage{amsmath,amssymb}
  \usepackage{url}   

  \usepackage{epstopdf}
\usepackage{hyperref}
\usepackage{upgreek}
\usepackage{comment}
\usepackage{pdfsync}
\usepackage{tikz-cd}

 \def\R{\mathbb R}

 \def \E {\mathbb E}

 \def \S {\mathbb S}

\def\eps{\epsilon}

\def \nbhd { neighborhood }

\newcommand{\beq}{\begin{equation}}
\newcommand{\eeq}{\end{equation}}
\newcommand{\Leq}[1]{\label{#1}\end{equation}}

\newtheorem{theorem}{Theorem}[section]
\newtheorem{proposition}{Proposition}[section]
\newtheorem{corollary}{Corollary}[section]
\newtheorem{lemma}{Lemma}[section]
\newtheorem{definition}{Definition}[section]
\newtheorem{remark}{Remark}[section]

\newcommand{\A}{\mathcal A}
\newcommand{\ey}{\frac{1}{2}}

\newcommand{\gm}{\gamma}

 \author[Richard Montgomery]{Richard Montgomery}
\address{University of California, Santa Cruz, Santa Cruz, CA 95064, USA }
\email{rmont@ucsc.edu}
\author[Guowei Yu]{Guowei Yu}
\address{Chern Institute of Mathematics and LPMC, Nankai University, Tianjin, 300071, China}
\email{yugw@nankai.edu.cn}

\thanks{ The second author is supported by NSFC (No. 12671224), Nankai Zhide Fundation and the Fundamental Research Funds for the Central Universities.}

\begin{document}

\title{Approach to total collision in the spatial 4-body problem}

\begin{abstract}
Every solution for the classical four body problem in space which
ends in total collision has a limiting shape.  We prove
that when  that shape is  planar  the   entire solution lies in a fixed plane and when   that shape is collinear
  the  entire solution lies in a fixed line. 
\end{abstract}

\maketitle

\section{Results} 

Every  total collision for the classical   four-body problem has a limiting shape.  That shape is one of a finite number of   central configurations [CCs]. 
See \cite{Hamp} and \cite{Albouy} regarding their  finiteness and   \cite{MoeckelCC} or chapter 1 of the book \cite{Mont}   for a 
general treatment of CCs.


 We   will  often be talking  about the  general  $N$-body problem in $d$-space. 
  An $N$-body configuration will always be  assumed   to have its 
 center of mass at the origin of the $d$-dimensional Euclidean space in which the $N$ point masses lie.   We call the  configuration   ``collinear '' if  the linear span of the  $N$  vectors describing the locations of the $N$ bodies  
  is 1-dimensional. We call the configuration ``planar'' if this span is 2-dimensional, and we call the configuration  ``spatial'' if this span is 3-dimensional.  
   We will call a solution to the $N$-body problem  `collinear' if at every instant all  $N$ bodies lie in a fixed line and we will
   call the solution `planar' if at every instant
   the $N$ bodies lie in a fixed plane.   Here, then,  is our main result:  
    \begin{theorem} \label{thm;planar}  Consider a   total collision  solution for the spatial 4-body problem.
    If the solution's   limiting shape is planar then the solution is planar.  If the solution's limiting shape is collinear then
    the solution is collinear.   
\end{theorem}

There is only one spatial CC shape for the 4-body problem: the regular tetrahedron. 
It follows from the theorem that if a total collision solution for the 4-body problem passes at some instant through a
spatial configuration then that solution's limiting shape is the regular tetrahedron.

This theorem  generalizes a well-known  result for the planar 3-body problem   
asserting that any total collision solution for that problem   whose limiting shape is collinear
is a collinear solution: all 3 bodies lie on a fixed line for all times up to collision.  That collinear theorem
holds in greater generality. 

 \begin{theorem} \label{thm;collinear} Consider a   total collision  solution for the  N-body problem in $d$-dimensional Euclidean space.
    If the solution's   limiting shape is collinear then
    the solution is collinear.   
\end{theorem}

Our initial motivation behind these theorems was variational. We can use the direct method of the calculus of variations to minimize action from any given  configuration (which is not a total collision) to a total collision, thus arriving at a total collision solution starting from the given configuration. If that fixed configuration is spatial then, by what we have just said, this action minimizer must tend to a total collision along a regular tetrahedral shape.  This variational thinking led us to the following result. 

\begin{theorem}
\label{thm;mini-coll} In the spatial four-body problem, given a positive time $T$ and a spatial configuration $q_0$, there exists at least one total collision solution $q(t)$ starting at  $q_0$ at $t=0$ and approaching  total collision in the shape of a regular tetrahedron as $t$ goes to $T$. 

Moreover for any $t \in (0, T)$, $q(t)$ must be a spatial configuration with the same orientation as $q_0$ and so the limiting regular tetrahedron 
also has that orientation.  
\end{theorem}

For the planar three-body problem, we have the following analogous result.

\begin{theorem} \label{thm;three-body}
  For the planar three-body problem, given any positive time $T$ and planar configuration $q_0$, there exists at least one triple collision solution $q(t)$, which starts from $q_0$ at $t=0$ and approaches a triple collision in the shape of an equilateral triangle as $t$ goes to $T$.  

  Moreover for any $t \in (0, T)$, $q(t)$ must be a planar configuration with the same orientation as $q_0$ and the limiting equilateral triangle at triple collision has the same orientation. 
\end{theorem}

\begin{remark}
Theorem  \ref{thm;three-body} was given in \cite{Hu11}. However the proof there seems to be incomplete, as in the proof of Lemma 4 in \cite{Hu11}, the author forgot to address the case that the minimizing solution may approach triple collision along one of the Euler configurations. Of course if the initial configuration is planar, this can be easily ruled out by Theorem \ref{thm;collinear} or its well-known version for the planar three-body problem. 
\end{remark}

{\it Sketch  of the proof of Theorem \ref{thm;planar}.} McGehee blow-up provides the framework. This blow-up adds a collision manifold as a boundary to the usual N-body phase space and extends the N-body flow to the collision manifold.   Equilibrium points of the extended  flow exist, they all lie on the collision manifold, 
and they correspond in a 2-to-1 manner to CCs.  
The collision manifold is a fiber bundle over the sphere of normalized configurations. Equilibria lie over the normalized CCs, with a pair of  equilibria  over each CC. One   equilibrium from this pair represents families of   solutions  incoming to total collision along that   CC, while the other   represents families of  solutions exploding out from total collision along that CC. Since we end in total collision,  we focus attention on the incoming
equilibrium of the  pair.
 
Assume now  that the CC is planar. The equilibrium may then  be viewed as an equilibrium for  the planar sub-problem of the spatial 4-body problem
in which all four bodies move in a fixed $\R^2 \subset \R^3$. 
This suggests a splitting of the domain of the linearization into `horizontal' and `vertical' where horizontal refers to
vectors arising from  variations of the CC which lie in the plane of the CC and vertical refers to vectors arising from 
variations perpendicular to the plane of the CC.    Central to our proof is  the fact that the ``vertical Hessian'' of the normalized potential ($\hat U$ below) at   a  planar  4-body  CC is negative semi-definite.
See definition \ref{def: vert Hess} below.  This fact is recorded in Moeckel's Lecture notes \cite{MoeckelCC}, as Prop. 20 on page 39,
and the analogous fact for collinear CCs as Prop. 19 on page 37. 
 Plugging this vertical Hessian information  into a  formula due to Devaney  and generalized by Moeckel  associates to each   {\it negative eigenvalue}  of the vertical Hessian a pair of   eigenvalues for the linearized flow whose real parts are   {\it positive}, and whose  eigenvectors 
 are `vertical' in the sense of being perpendicular to the invariant submanifold associated to the planar sub-problem. 
And  to each zero eigenvalue for the vertical Hessian corresponds one zero eigenvalue    and one unstable eigenvalue for the linearized flow. 

One also shows, and this is stated in the  propositions of Moeckel's lecture notes, that the zero-eigenspaces of the linearized flow
which arise  from the kernel
of the vertical Hessian arise  {\it entirely from symmetries}, meaning  from infinitesimal rotations whose rotation axes lie in the plane. 
 This suggests enlarging the planar sub-problem by taking its $G$-orbit in phase space,  where $G$ is the rotation group.  
 The resulting subvariety of phase space, is not smooth, but rather is the closure of a smooth variety we denote by   $\mathcal P (2)$,  (see \eqref{eq: subvariety} below) and which  contains the invariant subspace for our planar sub-problem. This subvariety   consists of 
  the set of all phase points whose
 ``space of motions'' (in the sense of Albouy-Chenciner) has dimension $2$ or less, while $\mathcal P(2)$
 is the set of all phase points whose space of motions has dimension exactly $2$.  This subvariety is a manifold, namely $\mathcal P(2)$,
  near the orbit of our equilibrium
 and it is foliated by the invariant planar subproblems consisting of motions where all bodies move in   $g \R^2 \subset \R^3$,  $g \in G$.
 
 The eigenvalue information yields that the linearized flow is normally hyperbolic to $\mathcal P(2)$
 along the $G$-orbit of our equilibrium,  and that  all the eigenvalues of this normal linearization having positive real part,
 so that the  flow just barely off of $\mathcal P(2)$   pushes away from $\mathcal P(2)$.  The proof is finished by using a geometric-dynamic  proposition,
 Proposition \ref{prop: repulsion}  regarding this situation.   
 
 We are grateful to the anonymous referee for suggesting the geometric-dynamic
 proposition which took the place of an alternative, but less transparent proof based on  center-stable manifold considerations.   

\section{Recalling Blow-up.} 


We begin by recalling the McGehee variables, and what the N-body 
equations   look like in these variables.
Take    $m_a > 0$  to be the masses where the index $a$ runs from $1$ to $N$. 
Write   $q = (q_1, \ldots, q_N)$, with   $q_a \in \R^d$ 
representing the position of the $a$th body.   We take the dimension $d$ general because
we will need to  compare phase space dimensions associated to  the spatial ($d=3$), planar ($d=2$)
and collinear ($d = 1$)  four body problems. 
Write  $v = \dot q : = dq/dt$ for the corresponding vector of velocities.  The kinetic energy of a motion is 
  $\frac{1}{2} \langle v , v \rangle  = \frac{1}{2} \sum m_a v_a \cdot v_a$ when 
we use the mass inner product     $\langle q,  y \rangle = \sum m_a q_a \cdot y_a$.
Newton's equations are $\ddot q = \nabla U(q)$
where  $U$ is the negative of the usual potential, where the gradient is with respect to the mass inner product,  and where the total energy
is $H(q, v) = K(v) - U(q)$.    Because of the translation invariance of the potential we can,
as is usual, assume that $\sum m_a q_a = 0$ and that  $\sum m_a v_a = 0$
throughout the motion.  In this way our  phase space of initial conditions, the  $(q, v)$'s 
has dimension $2 d(N-1)$ and is equal to   $\E_0 \times \E_0$ (minus collisions) 
where $\E_0$ is a Euclidean vector space of dimension $d(N-1)$.   

Set $$r = \|q \|: = \sqrt{ \langle q, q \rangle}$$
and $$q = r s, \text{ with }  s \in \S: = \{q:   \|q \| =1 \} \subset \E_0, $$
In English, $\S$ is the unit sphere of $\E_0$ relative to the mass metric.
Define a rescaled velocity $y$ by  $v = r^{-1/2} y$.   With respect to the combined position-velocity scaling
the total energy $H$ becomes homogeneous of degree $-1$: 
\begin{eqnarray*}
H(q,v)   &= & \frac{1}{r} (K(y) - U(s)) \\
  & = & \frac{1}{r} \tilde H(s, y),
\end{eqnarray*}
thus defining the homogenized energy $\tilde H$.  
Now decompose $y$ into a term along $s$ and a term
$w$ orthogonal to $s$:  
$$y = w + \nu s \text{ with } \nu = \langle s, y \rangle.$$ 
The  blown-up variables are  $(r, \nu, s, w)$. 
We have $r, \nu \in \R$ while the pair  $(s,w) \in T \S$
because  $w \perp s$.  
Rescale the    time variable  $t$ of Newton's equations to a new parameter $\tau$
 according to  $\frac{d}{d \tau} = r^{3/2} \frac{d}{dt}$. Write  $^{\prime}$ for $\frac{d}{d \tau}$.
 Write $\tilde H$ for $r H(q, v) = K(y) - U(s)$.
 This change of dependent  and independent variables
 transforms Newton's equations into  
\begin{equation}
\begin{aligned}
r^{\prime}& =  r \nu, \\
\nu^{\prime}& = \tilde H + \frac{1}{2} \|w \|^2 , \\
s^{\prime} & =  w,  \\
w^{\prime} & =     \nabla ^T U (s) - \frac{1}{2} \nu w - \|w \|^2 s .
\end{aligned}
\label{blowup}
\end{equation}
(Compare with  equation (70) page 77 of \cite{MoeckelClass}.) 
In the second of these equations,  the one for $\nu'$, the term  $\tilde H$ is  the homogenized energy introduced above. In   terms 
of the current variables we have
$\tilde H =  \frac{1}{2} (\nu^2 + \| w \|^2 ) - U(s)$. 
 In the last of these equations, the one for $w'$, the symbol   $ \nabla ^T$ denotes the covariant derivative operator on  the sphere $\S$,
 relative to the metric on $\S$ induced by the mass metric.  
Specifically   $  \nabla ^T   U (s) = \nabla U (s) - \langle s, \nabla U(s) \rangle s $ which we can also write
as $\nabla ^T   U (s) = \nabla U (s) + U(s)   s$ by using Euler's identity for  functions such as our $U$ which are homogeneous of degree $-1$. 
(For a derivation of this $\nu'$ equation, see the penultimate in the series of equations
for $\nu'$ found on p. 93 of  \cite{Mont}.)  
 The collision manifold is given by $r =0$, while
 the usual N-body phase space corresponds to $r > 0$.
 The main point of blow-up is that the ODEs make sense on the collision manifold.  
 
Equilibria $(r,\nu, s, w)$  for  the extended ODE  \eqref{blowup}    satisfy $r = 0$, $w = 0$, 
 and  $\nabla ^T U (s) =0$.  This last equality by itself characterizes
 the CCs  $s = s_0$  which are normalized.  To obtain the value   of $\nu$ at an equilibrium  associated to the  CC $s_0$ we  use
 the fact that the energy $H$ is conserved along  solutions away from  collision.   Take that value of the conserved energy to be a constant
 $h$.  From $\tilde H = r h$ we see that $\tilde H = 0$ at equilibrium or that   $\frac{1}{2} \nu^2 = U(s)$. 
 Now $U > 0$ everywhere so we have two corresponding roots   $\nu = \pm \sqrt{2 U(s_0)}$ 
 which correspond to the 2:1 relation between equilibria and CCs referred to above.  The negative root 
$ \nu_0 =-\sqrt{2 U(s_0)}$ corresponds to solutions incoming    to total collision ($r' < 0$)  and are the 
equilibria we are interested in.  {\bf Henceforth we identify a  normalized CC with its incoming equilibrium point:} 
\beq
s_0   \leftrightarrow (0, \nu_0, s_0, 0) \text{ with }   \nu_0 = - \sqrt{2 U(s_0)}.
\Leq{equil}

\section{Linear analysis at Equilibria} 
\subsection{Setting up} 
When we linearize the blown-up flow \eqref{blowup}   at  an equilibrium $s_0$  (see \eqref{equil}) we arrive at   the   linear system : 
\begin{equation}
\begin{aligned}
\delta r^{\prime}& =   \nu_0 \delta r, \\
\delta \nu^{\prime}& = \nu_0 \delta \nu  \\
\delta s^{\prime} & =  \delta w,  \\
\delta w^{\prime} & =   \tilde D \nabla ^T  U (s_0) \delta s - \frac{1}{2} \nu_0 \delta w
\end{aligned}
\label{linear_blowup}
\end{equation}
(Compare equation (78) of \cite{MoeckelClass}.) 
Here  
$(\delta r, \delta \nu, \delta s, \delta w)$ is a tangent vector to the full phase space at the equilibrium and represents a 
small perturbation  away from   equilibrium.    
We  require that $\delta s, \delta w$  both  be 
vectors in $T_{s_0} \S = s_0 ^{\perp}$  since  the linearization  of the   constraints $\langle s, s \rangle = 1$ and
$\langle s ,w \rangle = 0$ at the point  $(s, w) = (s_0, 0)$ yield  the constraints $\langle s_0,  \delta s \rangle = 0$ and
$\langle s_0, \delta w  \rangle = 0$. In arriving at the second equation above  we used  
$d \tilde H (s_0)  (\delta r, \delta \nu, \delta s, \delta w) = \nu \delta \nu + w \cdot \delta w -  d U(s) \cdot \delta s $
and that $dU(s_0) \delta s = 0$ since $s_0$ is a $CC$.  
The term $\tilde D   \nabla ^T  U (s) $ requires some explanation and care.
As an operator on $\E_0$ it is $D \nabla U (s_0) + U(s_0) I$.  We then restrict this operator to $s_0 ^{\perp}$.   More intrinsically, any  CC 
  is a critical point of $\hat U: \S \to \R$, the restriction of $U$ to our sphere $\S : = \{ r = 1 \}$.
It follows that   Hessian $d^2 \hat U (s_0)$ of $\hat U$ at our CC $s_0$ is a well-defined symmetric quadratic form on $T_{s_0} \S = s_0 ^{\perp}$.
The restriction of the inner product to $T_{s_0} \S$  is again an inner product on $T_{s_0} \S$ , still denoted $\langle \cdot , \cdot  \rangle$.
We can alternatively define  $\tilde D \nabla ^T  U (s_0) $  as the symmetric linear operator on $T_{s_0} \S$ such that   
$$ d^2 \hat U (s_0) (v, v) = \langle \tilde D \nabla ^T  U (s_0) v, v \rangle \text{ for  all } v \in T_{s_0} \S$$

Write $L$ for the   linear operator defined by  the right hand side of the linearized flow equation \eqref{linear_blowup}.
{\bf The   proofs  of our theorems   will  boil down to an understanding of the spectral properties of $L$
and those of its restrictions to  certain  invariant subspaces.} 

For future use it will  help to have  a symbol for the domain of
$L$.  We will use $T_{s_0} FULL$ for that domain so that 
\beq L:  T_{s_0} FULL \to T_{s_0} FULL   \text{ where }  T_{s_0} FULL
=
\mathbb R^2\oplus s_0^\perp\oplus s_0^\perp
\Leq{eq: domain L}
We use the tangent space notation since   the  domain of $L$  is the tangent space to the   phase space $FULL$  for our blown-up equations \eqref{blowup} 
\beq FULL  = \R^2_+ \times T \S
\Leq{eq: full}
with variables  $(r, \nu, s , w)$ as indicated in that equation and with 
$\R^2_+ = \{ (r, \nu) \in \R^2 :  r \ge 0 \}$.
$FULL$  is a manifold with boundary $r=0$ and its tangent space at our equilibrium is as indicated.

\subsection{Spectrum of the linearization}
We read off that $\nu_0$ is an eigenvalue for $L$ with  two-dimensional  eigenspace spanned by the
variation vectors of the form  $(\delta r, \delta \nu, 0, 0)$.  The other eigenspaces for  $L$ lie
in the complementary subspace of variation vectors of the form $(0,0, \delta s, \delta w)$.
This complementary subspace forms a hyperplane in the tangent  space to the collision manifold $r =0$ at our equilibrium point
and  is isomorphic to $s_0 ^{\perp} \oplus s_0 ^{\perp}$.   
The restriction $\hat L$ of $L$ to this complementary subspace   is given by  the matrix 
$$ \hat L = 
\left(
\begin{array}{cc}
  0 & I \\
\tilde D  \nabla ^T  U (s_0) & -\frac{1}{2} \nu_0  
\end{array}
\right)
$$
acting on column vectors  $(\delta s, \delta w)^T$.    
The following lemma on the spectral structure of $\hat L$ is due to Devaney and was clarified by Moeckel.
See Lemma 7.1 page 93 of \cite{MoeckelClass}. 
\begin{lemma}There is a 2:1 correspondence between the eigenvalues and eigenvectors of $\hat L$
acting on $s_0^{\perp} \oplus s_0 ^{\perp}$ and those of the symmetric Hessian operator $\tilde D \nabla ^T U (s_0)$ acting on $s_0^{\perp}$.  Under this correspondence an eigenvector $e$ 
for the Hessian operator having eigenvalue $a$ corresponds to the pair of linearly independent 
eigenvectors $(e, \lambda_+ e), (e, \lambda_- e)$ for $\hat L$ having eigenvalues $\lambda_+$
and $\lambda_-$ where 
\beq
\lambda_{\pm} (a) = -\frac{\nu_0}{4} \pm  \frac{1}{4} \sqrt{ \nu_0 ^2 + 16 a} 
\Leq{evalues}
These eigenvector  pairs  
issuing forth from the  eigendecomposition of the Hessian  
exhaust the eigendecomposition of $\hat L$.
\label{spectral lemma}
\end{lemma}

Applying the eigenvalue and vector formulae of the above lemma and recalling that $-\nu_0 > 0$  yields: 
\begin{corollary} \label{cor;invariant}  If the subspace  $S \subset s_0^{\perp}$ is   invariant under the Hessian operator and  if   the Hessian is negative definite on $S$
 then the subspace $0 \oplus S \oplus S \subset T_{s_0} FULL = \R^2 \oplus s_0 ^{\perp} \oplus s_0 ^{\perp}$ is
   invariant under $L$ (as above) and all the eigenvalues of  the restriction of $L$ to this subspace  have strictly positive real part.  
 
  If  the subspace  $S \subset s_0^{\perp}$ lies in the kernel of the Hessian operator 
 then $0 \oplus S \oplus 0 \subset T_{s_0} FULL$ lies in the kernel of $L$
 while   $0 \oplus graph (-\nu_0/2 |_S) : = \{ (0, s , \frac{-\nu_0}{2} s): s \in S\}$  is contained in the $-\nu_0/2$ eigenspace of  $L$\}.  Note: $- \nu_0 / 2 > 0$
 \end{corollary}

\begin{remark} [On fixed energy and the homothetic solution]
Fixing the  total  energy    $H(q, v) = h$ isolates the one-dimensional subspace $h \delta r = \nu_0 \delta \nu$
of the $(\delta r, \delta \nu, 0, 0)$-plane,   the $\nu_0$ eigenspace for $L$.
Indeed, use $\tilde H = r H$
to arrive at the  energy constraint $r h = \frac{1}{2} (\nu^2 +  \|w \|^2) - \hat U (s)$, which, upon   linearizing 
 at our equilibrium yields  $h \delta r = \nu_0 \delta \nu$.  This  one-dimensional subspace of the plane  is tangent to the
 incoming homothetic solution associated to $s_0$ and having energy $h$.   
\end{remark}

\subsection{ Symmetries and Zero eigenvalues} \label{ss: symmetry} 
The potential $U: \E_0 \to \R$ is invariant under the action of the rotation group $G$ of $\R^d$,
where an element $g \in G$ acts on $\E_0$ by sending  $q = (q_1, \ldots, q_N) \in \E_0$ to  $g q = (gq_1, \ldots, gq_N)$.   It follows
that $d^2 \hat U (s_0)$ and hence $\tilde D \nabla ^T  U (s_0)$ have a nontrivial  kernel, 
this kernel containing ${\mathfrak g}(s_0)$, the tangent space to the rotational orbit through $s_0$.
Here  ${\mathfrak g}$ denotes the Lie algebra of the rotation group so that
$$\mathfrak g (s_0) = \{ As_0:   A  \text{ a skew symmetric operator on  } \R^d \}.$$
where $As_0 = (A s_{01}, A s_{02}, \ldots, A s_{0N})$.

This kernel represents the zero eigenvalue
$a = 0$ for the Hessian operator  in formula \eqref{evalues} and so yields   two eigenvalues
(formula \eqref{evalues})   for $\hat L$, namely $\lambda_+  = - \nu_0/2$ and $\lambda_- = 0$.
Since we   take $\nu_0 < 0$  we have that  $\lambda_+ > 0$.  
\begin{definition} We call the CC $s_0$ nondegenerate if the kernel of $d^2 \hat U (s_0)$
is equal to ${\mathfrak g}(s_0)$. 
\end{definition}

The orbit  $Gs_0$ of a nondegenerate CC is isolated from all other CCs.  
For almost all mass ratios, all CCs for the four-body problem are nondegenerate.
However degenerate CCs do occur at special mass ratios. 

$G$ also acts on phase space according to $g(q, v) = (gq , gv)$ and so
on the blown-up variables by $g(r, \nu, s, w) = (r, \nu, gs, gw)$.  
When we view $s_0$ as an  incoming equilibrium as in  \eqref{equil}  then its  $G$ orbit
is $G s_0 = \{ (0, \nu_0, gs_0, 0) : g \in G\}$ which has tangent space $\mathfrak g (s_0) = \{  (0, 0, A s_0, 0):  A \in \mathfrak g \}$.
The elements $(0, 0, A s_0, 0)$  correspond to the zero eigenvalue for $L$. It follows 
this   $\mathfrak g (s_0)$  lies in the kernel of $L$ and forms its entire  kernel if and only if
$s_0$ is nondegenerate.

 \subsection{Invariant subproblems and the key Proposition}    For a configuration $q \in \E_0$
 write $span(q) \subset \R^d$ for the span of the components $q_a \in \R^d$ and
 write $dim(q) = dim(span(q))$.  Suppose  our central
 configuration  $s_0$  satisfies $dim(s_0) = k$ where  $k < d$.  
 Choose orthonormal coordinates for $\R^d$ such that
   $span(s_0) = \R^k$ is  the span of the first k standard basis elements while 
  $\R^{\ell}$ is  the orthogonal subspace to   $\R^k$ within $\R^d$, so
 that $$\R^d = \R^k \oplus \R^{\ell}$$ and $k + \ell  = d$.   
 
 The condition that $dim(s_0) = k < d$,  singles out an 
 invariant subproblem of our N-body problem, namely the problem in  which all bodies move within our fixed $\R^k = span(s_0)  \subset \R^d$.
 We will refer to the associated invariant submanifold of the full phase space as ``the horizontal subproblem'',
and to  $\R^k$ and $\R^{\ell}$ as being  `horizontal' and `vertical' spaces. 
The horizontal  subproblem and its associated decomposition of motions near $s_0$ will guide our thinking.
 
 \begin{definition} Under the assumption above that
 $span(s_0)= \R^k \subset \R^d$ we call a variation $\delta s = (\delta s_1, \ldots , \delta s_N) \in s_0 ^{\perp}$
 ``vertical'' if  each $\delta s _a \in \R^{\ell} = span(s_0)^{\perp}$ and write $V$ for the vector space of vertical variations.
 We call a variation  `horizontal' if  each $\delta s_a \in \R^k$ and write $Hor$ for the space of horizontal variations. 
 \label{def: vert}
 \end{definition}

\begin{figure} 
\centering
\includegraphics[scale=0.6]{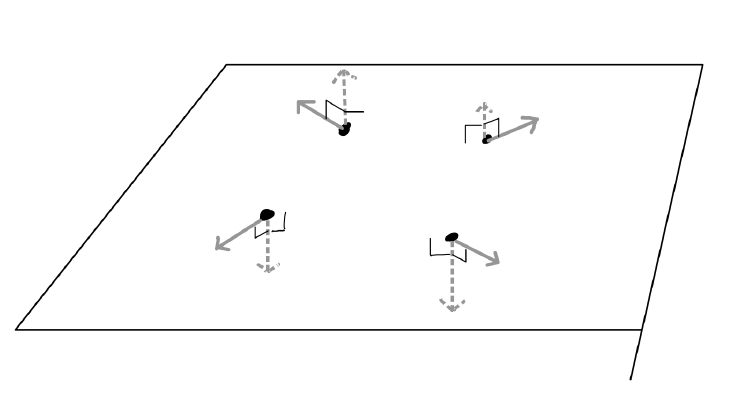}
\caption{Vertical and horizontal variations of a planar 4-body configuration.}
\label{fig;Hor-Vert}
\end{figure}
 
 \begin{remark}
 Our use of horizontal and vertical here is not to be confused with the 
 horizontal and vertical spaces associated to the condition `angular momentum equals zero'
 and connections on principal bundles    frequently used in   other papers by one of us.
 We do not use `$H$' for $Hor$ so as to not confuse this subspace with a Hamiltonian.   
 \end{remark}
 
 \begin{lemma}  The space $s_0 ^{\perp}$ on which our normalized Hessian
 acts is the direct sum of the subspace of vertical variations and the subspace  of  horizontal variations.
 In symbols, 
 \beq
 s_0 ^{\perp} = Hor \oplus V
 \Leq{decompQ}
 Each  subspace $Hor$ and $V$ is   invariant under the action of the Hessian operator.
 \end{lemma}
 
 Initial rough computations suggest that the Hessian of $\hat U$ is negative definite on $V$.
 These computations are wrong, but  they point us  in the right direction.  
 \begin{definition}
The ``vertical Hessian'' at a CC $q$ is   the restriction of the Hessian (either as a quadratic form or as a symmetric operator)
of  $\hat U$ 
to $V$.  The   vertical index of $q$ is   the index of the vertical Hessian at $q$.
Let us write $vert ind (q)$ for the vertical index at $q$.
\label{def: vert Hess}
\end{definition}

  The  vertical index is always positive.   
 Indeed,  write $codim(s_0)$ for $d - dim(s_0)$ and assume
 $dim(s_0) < N-1$.  In the  proof of Proposition 20 of \cite{MoeckelCC} it is shown that this
 vertical index is at least  $codim(s_0)$ provided $dim(s_0) < N-1$.
 Our work here concerns the case where the vertical index is, roughly speaking, ``as large as possible''.

 \begin{definition}  We say that a central configuration is `vertically regular' if
 its normalized Hessian $d^2 \hat U$ is negative semidefinite on vertical variations, and if, 
 moreover,  the kernel of restriction of the Hessian operator to vertical variations  
 coincides with the space  ${\mathfrak g} (s_0) \cap V$ of `vertical symmetries'. 
 \end{definition}
 
We can always split  
$V$ up  into  the subspaces of zero, positive, and negative eigenvalues for 
the vertical Hessian, writing this splitting
as 
\beq
V = V^0   \oplus V^{-} \oplus V^{+}.
\Leq{decompV}
We then have that $vert ind (q) = dim(V^{-})$. 
To say that $q$ is vertically regular is the same as saying
that $V^+ = 0$ and $V^0 =  {\mathfrak g} (q) \cap V$.

 The following two propositions are proved by Moeckel \cite{MoeckelCC}.
 \begin{proposition} Every collinear CC is vertically regular.
 \label{propC}
\end{proposition} 

\begin{proof}
This is the assertion of proposition 19 of p. 37 of \cite{MoeckelCC}.  
\end{proof}

\begin{proposition} If $q$ is a CC for the $N$-body problem in $d$-space 
having $dim(q) = k$ and if $(k, d, N) = (N-2, N-1, N)$ then $q$ is vertically regular.
\label{propA}
\end{proposition} 

\begin{proof} Proposition 20 in \cite{MoeckelCC} asserts that for general $d, k, N$
with $N > d > k$ 
we have  $ind(q) \ge d-k$.     A reading of the proof of proposition 20 
shows that   Moeckel actually proves the stronger inequality:
$vertind(q)  \ge d-k$.   Under
the hypothesis of  proposition \ref{propA} we have   $d- k = 1$
while $dim(V) = N-1 = d$.  (The dimension is $N-1$ rather than $N$
due to the condition $\sum m_a \delta s_a = 0$.) 
 The proof of the proposition will be  completed upon proving 
 \beq
dim({\mathfrak g} (q) \cap V) \ge  d- 1.
\Leq{vertsymmdim}

To see that inequality \eqref{vertsymmdim} completes
the proof use the decomposition \eqref{decompV}  
and recall from  subsection \ref{ss: symmetry}  that   ${\mathfrak g} (q) \cap V \subset V^0$.
By $vert ind (q) \ge 1$ we have    $dim(V^{-})  \ge 1$.
Thus   \eqref{vertsymmdim} combined with the observation that
$(d-1) + 1 = d = dim(V)$   proves that $V^0 =  {\mathfrak g} (q) \cap V$
and   $V^+ = 0$ which are the conditions of $q$ being vertically regular. 

To prove \eqref{vertsymmdim} recall that for $v, w \in \R^d$
we get  $v \wedge w \in \mathfrak g$, this being  the skew-symmetric linear operator on $\R^d$  which sends $x \in \R^d$ to $(v\wedge w) (x): = \langle v , x \rangle w - \langle w, x \rangle v$.
Now let $e_d$ be  the unit normal to the hyperplane
$span(q) \subset \R^d$ and let  $e_i, i = 1, \ldots d-1$
be an orthonormal  basis for this hyperplane.  Consider the $d-1$ elements $A_i$ of ${\mathfrak g}$ given 
$A_i = e_i \wedge e_d$, $i =1, \ldots d-1$.   Then
$A_i (q)$ has $a$th component $A_i (q_a)=   \langle e_i, q_a \rangle e_d$ for $a = 1, \ldots , N$
so that each $A_i (q)$
  is a vertical variation.  Since $span(q) = span(e_1, \ldots , e_{d-1})$
the matrix $\langle e_i, q_a \rangle$ has rank $d-1$ from which it follows
that the collection    $A_i (q)$   in $\R^d \otimes \R^N$, $i =1, \ldots , d-1$ are linearly independent and thus 
span a space of dimension $d-1$ within ${\mathfrak g} (q) \cap V$.  Since each  $A_i (q) \in {\mathfrak g}(q) \cap V$ we have shown that  $dim( {\mathfrak g} (q) \cap V) \ge d-1$,
as required.
 
\end{proof} 

 \begin{remark}
 CCs with $dim(q) < d$ are not always vertically regular. 
 See \cite{Moeckelbifurc} or \cite{MoeckelSimo} for examples of planar CCs for $N> 4$ bodies  in 3-space which are not vertically regular.
 For such  CCs the vertical Hessian has   positive   as well as negative directions. 
\end{remark} 

 \section{Proof of   theorems via a  crucial proposition}
 
 Write $G = SO(d)$ for the rotation group as it acts on configurations and on velocities
 and hence on the full phase space \eqref{eq: full} which we have denoted by $FULL$.   If $s_0$ is the equilibrium
 corresponding to an incoming  CC then   $G s_0$  is a $G$-orbit of incoming
 equilibria lying on the collision manifold (see  \eqref{equil}).  Solutions having their  $\omega$-limit
 set contained in  this orbit correspond to  those total collision solutions whose incoming shape 
 agrees with that of  $s_0$.  (We leave open the problem of `infinite spin' in this general setting.)
 
{\bf Proof of the main theorems.} 
Observe that proposition \ref{propA} applies to the case $(k, d, N) = (2, 3, 4)$
 showing that the planar 4-body CCs are vertically regular within the spatial 4-body problem. 
 It thus follows that  Theorems \ref{thm;planar} and \ref{thm;collinear} follow immediately from propositions \ref{propA}, \ref{propC}
 and the crucial:

 \begin{proposition}  
 Suppose that $s_0$ is a CC for the N-body problem in d-dimensional Euclidean space 
 with   $dim(s_0) = k  < d$ and with  $s_0$  being 
 vertically regular.  Then   any total collision solution $q(t)$  whose $\omega$-limit set after  blow-up  lies in the equilibrium orbit 
  $G s_0$ has  the property that all  the bodies $q_a (t) \in \R^d$ represented by the solution
  lie in a fixed k-dimensional subspace of $\R^d$ throughout the motion.
  \label{prop: main} 
 \end{proposition} 
   
The proof of this proposition is given in the next section.

\section{Proof of the  crucial proposition, Proposition \ref{prop: main}} 
 
 \subsection{Invariant subproblems and the key proposition} 
As usual, if $C $ is any set or list of   vectors
in $\R^d$ we will write $span(C) \subset \R^d$ for their span.  
A point $Z = (q, v)$ in our phase space is such a list so we can form
$span(Z) \subset \R^d$, a linear subspace.
\begin{definition} For $Z \in \E_0 \oplus \E_0$ 
we call $span(Z) \subset \R^d$ the space of motions associated to $Z$.
(This definition is taken  from \cite{AC}.) 
\end{definition} 
If $Z(t) = (q(t), v(t))$ is  a solution to the N-body equations then  
 $span(Z(t)) = span(Z(0))$.  This fact  follows from the fact that the force on the $a$th body,  $a = 1, \ldots , N$ necessarily
 lies in $span(Z(t))$.  It follows that when $span(Z)$ is a proper subspace of 
 $\R^d$ we have, associated to our initial condition, an   invariant sub-problem of the N-body problem,
 namely the problem in which all bodies travel in a fixed  $\R^k \subset \R^d$,  where $k = dim(span(Z))$.  
 \begin{definition} 
 Write $dim(Z) = dim(span(Z))$ for $Z = (q, v)$ 
 a phase point. 
 \end{definition}
 Thus, if $dim(Z(0)) = k < d$ we have seen that the consequent motion takes place
 in some fixed k-dimensional subspace $span(Z(0))$ of $\R^d$.   Such a subspace
 defines an invariant subproblem to our N-body problem, namely  the
 N-body problem in which the bodies move in $\R^k$.
 For this purpose imagine fixing a linear subspace $\R^k \subset \R^d$.
 Then (before blow-up) the  phase space for this invariant subproblem is
  $$\mathcal P (\R^k) = \{ Z = (q, v):  q_a, v_a \in \R^k \} \subset \mathcal P (\R^d)$$
 Note that there are points $Z \in \mathcal P (\R^k)$ for which $span(Z) \ne \R^k$
 by virtue of the fact that $dim(Z) < k$.  
  
 The following lemma is elementary but important.
 \begin{lemma}  If $Z(t) \in FULL$ is a solution curve for  the N-body problem then $span(Z(t))$ and hence  $dim(Z(t))$
 is independent of $t$ over any interval of time on which the solution is defined.
 \end{lemma} 
 Define the flow-invariant subvariety
 \beq
 \mathcal P (k) = \{ Z  \in \mathcal P (\R^d): dim(Z) = k \}
 \Leq{eq: subvariety} 
 By slight abuse of notation we use the same symbol $\mathcal P (k)$
for  the extension of  $\mathcal P(k)$ to  the total collision manifold.
Thus  $\mathcal P(k)$ consists of all points $(r, s, \nu, w)$ in $FULL$ for which 
  $dim(s, w) = k$.  Similarly, we use the same symbol $\mathcal P (\R^k)$
 for the extension of $\mathcal P (\R^k)$ to $FULL$ which is the set
 of all points $(r, s, \nu, w)$ in $FULL$ for which    $span(s, w) \subset \R^k$.  
 
 We have that $s_0 \in \mathcal P(\R^k) \cap   \mathcal P(k)$.
 This set   consists  of those  $(r, s, \nu , w)$ in $\mathcal P (\R^k)$ for which     $dim(s, w) = k$, 
 i.e. the maximal rank points.  Consequently this set is open and dense.
 
 \begin{lemma}
$(\rm a)$. $\mathcal P (k)$ is an smooth submanifold of   $FULL$. 
  
$(\rm b)$. $\mathcal P(k) =  G (\mathcal P (\R^k) \cap \mathcal P(k))$.
 \end{lemma} 

\begin{proof}  
(a). The next three paragraphs prove (a).   
Following Albouy and Chenciner \cite{AC} (see also \cite{Dullin_Montgomery})
we identify the phase space $\E_0 \times \E_0 = \E_0 \oplus \E_0$ {\it before blow-up} with a matrix vector space. This can be done, for example, as follows.
View N-body phase space, {\it before reduction by translation}, as   $\E = Hom(\R^N \oplus \R^N, \R^d)$. Here the first copy of $\R^N$ comes
with  canonical basis, $e_1, \ldots , e_N$ and we write the canonical basis for the second copy of $\R^N$ as 
$\epsilon_1, \ldots , \epsilon_N$.  The old  N-body phase point $(q, v)$ 
with   $q = (q_1, \ldots, q_N)$,  $v= (v_1, \ldots, v_N)$ and $q_a, v_a \in \R^d$  corresponds to
the linear map  $Z \in Hom(\R^N \oplus \R^N, \R^d)$ for which $Z(e_a) = q_a$ and $Z(\epsilon_a) = v_a,  a =1 \ldots, N$.
The  center-of-mass $q_{cm} = \frac{ \sum m_a q_a} {\sum m_a}$ and total linear momentum $p_{cm} = \sum m_a v_a$
together define a linear map $\E \to \R^d \oplus \R^d$   sending $Z$ as above to $(q_{cm}, p_{cm})$. 
 $\E_0 \oplus \E_0$ is  the kernel of this linear map.  This kernel  can be identified with
$Hom( D \oplus D,  \R^d)$ where $D$ is an $N-1$ dimensional real vector space  {\it having  no canonical basis}.
See \cite{AC} for details and for an intrinsic representation of  $D$   ``the space of dispositions''.  

Viewing  the pre-blow-up phase space $\E_0 \oplus \E_0$ this way, then $\mathcal P (k)$, {\it away from total collision},  consists of all those linear maps $D \oplus D \to \R^d$
whose rank is $k$.  By the lemma found on \cite[Section 2.2, p. 29]{Arnold12}   this space of matrices is a smooth real algebraic variety of codimension
$(2N -2-k)(d-k)$.  (We assume that $N \ge d > k$).

What about $\mathcal P(k)$ along the collision locus?  Why is it still a manifold, now with boundary?
We recall that $FULL$ near total collision  consists of $(r, \nu, s, w) = (0, \nu, s, w)$ for which 
$rank(s, w) = k$.   We revert then to the original version of  $\mathcal P(k)$ as consisting of those
quadruples $(r, \nu, s, w) = (0, \nu, s, w)$ such that $rank(s,w) = k$.  Consider the map $\hat Z: T \S \to \E_0 \oplus \E_0$ sending
$(s, w) \to \hat Z (s,w) = Z(s, w) \in Hom(D \oplus D, \R^d) \subset Hom(\R^N \oplus \R^N, \R^d)$, where the matrix $Z$ is understood
by the above choice of basis as discussed two paragraphs above.  This map $\hat Z$ is a smooth embedding, in matrix terms
just sending $(s,w)$ to a $d \times 2N$ matrix of the form $[s, w] \in Hom(\R^N \oplus \R^N, \R^d)$.  
Write $\mathcal Q (k) \subset Hom(D \oplus D, \R^d)$ for  those matrices whose rank is $k$.  By the lemma from Arnol'd et al \cite{Arnold12}, this is a smooth manifold.  By construction,  $\mathcal P (k) = \R^2 _+ \times \hat Z ^{-1} (\mathcal Q (k))$
since the rank of $(r,\nu, s ,w)$ is independent of $(r, \nu) \in \R^2 _+ $.  Thus to show we have a manifold
it is enough to show that the inverse image of $\mathcal Q (k)$ under the embedding $\hat Z$
is a manifold. We do this by {\it identifying} $T \S$ with its image under $\hat{Z}$ and instead looking
at the defining equations of $T \S$ within $Hom(D \oplus D , \R^d)$. 
Writing $[q, v]$ for a general point of $Hom(D \oplus D , \R^d)$.  Then $T \S$ is
defined by the  two   equations,
$\langle q, q \rangle =1$ and $\langle q, v \rangle = 0$.   The associated scalar functions
$f(q, v) =  \|q \|^2$ and $g(q, v) = \langle q, v \rangle$ define a map
$(f, g): Hom(D \oplus D , \R^d) \to \R^2$ and the point $(1,0)$ is easily checked to be a regular value for this function,
which is to say that $T \S$ is the inverse image of $(1,0)$ under this map.   Now suppose
that $[q,v] = [s, w]$ lies on $T \S$ and also  has rank $k$.  Consider the two-parameter
perturbation $[s, w] \mapsto [\lambda s,  w + \alpha s]$ of our point $[s, w]$.   The deformation  preserves rank,
so defines a two-parameter variation lying entirely  within the smooth submanifold $\mathcal Q(k)$.  The deformation
passes through $[s, w]$ when $\lambda = 1, \alpha = 0$. Restricted to   this variation we have    $(f, g)([\lambda s,  w + \alpha s] ) = (\lambda^2, \lambda \alpha  )$
so that the differential of $(f,g)$ along such variations is  $(2 \lambda d \lambda,  \lambda d \alpha + \alpha d \lambda) = (2d \lambda,  d \alpha)$.  We have just shown that our $(s, w)$  is a {\it regular point}
for the restriction of the map $(f,g)$ to points of the smooth manifold $\mathcal Q (k)$, which shows 
 that $T \S \cap \mathcal Q (k)$ is a smooth manifold.
Our $\mathcal P(k)$ is this smooth manifold times $\R^2 _+$.
 This establishs (a).

(b).  Notice that rotations preserve rank. Hence $G\bigl(\mathcal P(\R^k)\cap\mathcal P(k)\bigr) \subseteq \mathcal P(k).$
Conversely, if $(r,\nu,s,w)\in\mathcal P(k)$, then
$\operatorname{span}(s,w)$ is a $k$-plane.
Choose $g\in SO(d)$ taking this plane to $\R^k$.
Then $g(r,\nu,s,w)\in\mathcal P(\R^k)\cap\mathcal P(k)$, and hence
$$ (r,\nu,s,w) =g^{-1}(g(r,\nu,s,w)) \in G\bigl(\mathcal P(\R^k)\cap\mathcal P(k)\bigr). $$ 
This establishes (b).
 
\end{proof}

\begin{remark} 
 $G (\mathcal P( \R^k))$ is a singular variety consisting of all those points $(r, \nu, s, w)$ for which $rank(s, w) \le k$. Its smooth locus is $\mathcal P(k)$ and it is also the closure of $\mathcal P (k)$.  
 \end{remark}   
 
 Observe that 
 \beq
  G s_0 \subset \mathcal{P}(k) = G (\mathcal P (\R^k) \cap \mathcal P(k))
 \Leq{eq: nesting}
 The following lemma will be needed.
 Since $\mathcal P(k)$ is a smooth submanifold of the full phase space $FULL$ we can form its normal bundle $\nu$
and thus get an orthogonal  splitting
\beq 
T_s FULL = T_s \mathcal P(k) \oplus \nu_s
\Leq{eq: split1}
which is valid all along $\mathcal P (k)$ and in particular for points $s \in Gs_0$.
\begin{lemma}
If $s_0$ is vertically regular,  and $s \in Gs_0$  then the linear map induced by
$L(s) = DY(s)$ on the normal quotient
$T_s FULL/T_s\mathcal P(k)$ 
has all its eigenvalues with strictly positive real part.
\end{lemma}

\begin{proof}
Use $G$-invariance to observe that it is enough to verify
the lemma at $s=s_0$. We use the original representation
$$
T_{s_0} FULL
=
\mathbb R^2\oplus s_0^\perp\oplus s_0^\perp
$$
for  the tangent space at our equilibrium, with tangent vectors written
$(\delta r,\delta\nu,\delta s,\delta w)$ and referred to
as variations, as earlier (see \eqref{linear_blowup}).
Thus $(\delta r,\delta\nu)\in\mathbb R^2$ and
$\delta s,\delta w\in s_0^\perp$.
Using the horizontal-vertical splitting
$$
s_0^\perp=Hor\oplus V,
$$
we observe that
$$
T_{s_0}\mathcal P(\mathbb R^k)
=
\mathbb R^2\oplus Hor\oplus Hor.
$$
Then $\mathcal P(k)=G (\mathcal P (\R^k) \cap \mathcal P(k))$ gives
$$
T_{s_0}\mathcal P(k)
=
T_{s_0}\mathcal P(\mathbb R^k)+T_{s_0}Gs_0
$$
(not a direct sum, just a sum of subspaces).
At our equilibrium point $s_0=(0,\nu_0,s_0,0)$, we have
$$
T_{s_0}Gs_0
=
0\oplus\mathfrak g(s_0)\oplus0 \subset \mathbb R^2\oplus s_0^\perp\oplus s_0^\perp.
$$
Write
$\mathfrak g(s_0)_V=\mathfrak g(s_0)\cap V$.
Since
$$
Hor+\mathfrak g(s_0)
=
Hor\oplus\mathfrak g(s_0)_V,
$$
we obtain
$$
T_{s_0}\mathcal P(k)
=
\mathbb R^2\oplus
\bigl(Hor\oplus\mathfrak g(s_0)_V\bigr)
\oplus Hor.
$$

Consider the linear transformation
$-\frac{\nu_0}{2}I$ on $\mathfrak g(s_0)_V$.
By slight abuse of notation, set
$$
\Gamma
:=
\operatorname{graph}
\left(-\frac{\nu_0}{2}I_{\mathfrak g(s_0)_V}\right)
=
\left\{
\left(0,u,-\frac{\nu_0}{2}u\right):
u\in\mathfrak g(s_0)_V
\right\}
\subset T_{s_0}FULL,
$$
where the first zero denotes the $\mathbb R^2$ component.
Since $\nu_0\ne0$, we have
$$
0\oplus\mathfrak g(s_0)_V\oplus\mathfrak g(s_0)_V
=
\bigl(0\oplus\mathfrak g(s_0)_V\oplus0\bigr)
\oplus\Gamma.
$$

By vertical regularity,
$$
V=\mathfrak g(s_0)_V\oplus V_-,
\qquad
V_-:=V\cap\mathfrak g(s_0)_V^\perp,
$$
where the Hessian is negative definite on $V_-$ and
$\mathfrak g(s_0)_V$ is the kernel of the vertical
Hessian. Now apply the horizontal-vertical splitting factor by
factor and the above equalities to get 
\begin{align*}
T_{s_0}FULL
&=
\mathbb R^2\oplus(Hor\oplus V)\oplus(Hor\oplus V)\\
&=
T_{s_0}\mathcal P(k)
\oplus\Gamma
\oplus(0\oplus V_-\oplus V_-).
\end{align*}
It follows that the normal bundle at $s_0$ is
$$
\mathcal \nu_{s_0}
:=
\Gamma\oplus(0\oplus V_-\oplus V_-).
$$ 

Finally it follows from the vertical regularity of $s_0$ and Corollary \ref{cor;invariant} that $\nu_{s_0}$ is invariant under $L(s_0)$,
and the restriction to it has all eigenvalues with
strictly positive real part.
Since $T_{s_0}\mathcal P(k)$ is invariant, the quotient
map identifies this restriction with the induced
linearization on
$T_{s_0}FULL/T_{s_0}\mathcal P(k)$.
\end{proof}

\subsection{A normal-repulsion proposition  }

 To finish off the proof, we apply the following proposition.
 To have the proposition fit our situation, take $Y$ to be the blown-up Newtonian vector field,  $M = FULL$
 to be the full blown-up phase space,  $\mathcal P = \mathcal P (k)$ 
 as per the earlier part of the section and $K = Gs_0$ in the proposition.  
 \begin{proposition} \label{lem: normal repulsion}
Let $Y$ be a smooth vector field on a smooth manifold $M$, possibly with boundary.  
Suppose that $K \subset {\mathcal P} \subset M$  are smooth submanifolds where $K$ is   compact and consists of zeros of $Y$  
  and where ${\mathcal P}$ is   invariant under the flow of $Y$. 
Suppose   the linearization   of $Y$ along $K$  is ``repelling   transverse to ${\mathcal P}$''
in the   sense that the    linear map induced by $DY(k)$  on the fibers    of  the normal bundle to ${\mathcal P}$ at points $k \in K$
   has all  eigenvalues with strictly positive real part.   Then there is a neighborhood $\mathcal U$ of $K$ such that every trajectory which remains inside   $\mathcal U$ for all forward times must belong to ${\mathcal P}$   for its maximal time of existence.
   \label{prop: repulsion}
\end{proposition}

\begin{proof}
We   construct a smooth  non-negative function  $\mathcal E \ge 0$ defined in a relatively compact \nbhd $U$ of K,
smooth  on its closure   $\bar U$,  
whose zero locus in $\bar U$  is precisely ${\mathcal P} \cap \bar U$  and which satisfies 
\beq
\mathcal E' \ge \frac{1}{\alpha} \mathcal E   ; \;   \alpha  > 0   \text{ constant }
\Leq{reverse Liap} 
along trajectories of $Y$ lying in $U$.  Here $\alpha$ is a constant and   $\mathcal E' = d{\mathcal E} (Y)$ denotes the time derivative of $\mathcal E$ along integral trajectories
$\gamma(t)$ for $Y$.   We call such a function a ``reverse Lyapunov function'' for ${\mathcal P}$. 

To see that the existence of such a reverse Lyapunov function $\mathcal E$ yields the proposition suppose that $\gamma(t)$ is a trajectory for $Y$ which enters $U$ but does not lie on ${\mathcal P}$  so that $\gamma(t_0) \in U$ for some time $t_0$ while $\mathcal E (\gamma(t_0)) > 0$.  Let $C = max_{m \in \bar U} \mathcal E (m)$
so that $\mathcal E(m) \le C$ for $m \in U$.
Compare $\mathcal E(\gamma(t))$ to  the solution   to $\mathcal F' = \frac{1}{\alpha} \mathcal F$
satisfying  $\mathcal F(t_0) = \mathcal E(\gamma(t_0))$.  Inequality \eqref{reverse Liap} implies that   for all $t > t_0$ such that
  $\gamma(s)$ 
stays in $U$ on the interval $[t_0, t]$ we have that $\mathcal E(\gamma(t)) \ge \mathcal F(t)$.  Now $\mathcal F(t) = \exp\left(\frac{t-t_0}{\alpha}\right)  \mathcal E(\gamma(t_0))$ so as   soon as $t >  t_*:=  t_0 + \alpha log(C/\mathcal E(\gamma(t_0))$ we have $\mathcal F (t_*) >  
 C$ which  implies that  $\mathcal E (t) > C$.
But $\mathcal E \le  C$ everywhere on $U$  showing that   our curve $\gamma(t)$ must have left   $U$ at some earlier  time $t< t_*$.  This shows that  any curve which enters $U$ and which never leaves $U$ must lie on the zero locus of  $\mathcal E$ which is precisely ${\mathcal P} \cap U$.  This is the assertion of the proposition. 

 The construction of our reverse Lyapunov function follows a basic construction from linear control theory which goes back to
 Lyapunov, see \cite{Sontag98}.    A real square matrix $A$  is termed ``Hurwitz'' if its spectrum lies strictly in the left half plane $Re(\lambda) < 0$.
 Associated to such an  $A$ we have the linear vector  ODE $z' = A z$ with   origin a global attractor.
 This basic  construction builds  a    quadratic positive-definite Lyapunov function $V(z) = z^t P z$ for the   ODE.
 (Think of   $z$ as a column matrix.)    To get the construction to work we require $P$ to be a   positive definite matrix $P$  satisfying  
 $A^t P + P A  = - I$. The crux is that this solution has a unique solution if and only if $A$ is Hurwitz.

Choose  a smooth Riemannian metric on $M$ near $K$.  Along with the metric we get the identification
of the fibers of the abstract normal bundle  
$T_pM/T_p\mathcal P$ to ${\mathcal P}$ with the geometric normal bundle   $\nu_p : = T_p {\mathcal P} ^{\perp}$.  We also get the  identification arising from the usual proof of the  tubular neighborhood
theorem of  a neighborhood of $\mathcal P$ in $M$  with a neighborhood
of the zero section of $\nu$ in $\nu$.     Now    locally trivialize $\nu$ over relative open sets    $W_i \subset P$ 
covering $K$, so that over each $W_i$ we have $\nu|_{W_i} \cong W_i \times \R^c$
and these form a system of coordinates of the form     $(x, z)$ with $x \in W_i$  and  $z \in \R^c$,  where $c$ is the codimension of
${\mathcal P}$ within $M$ which cover $K$.  (To fit inside our tubular \nbhd and thus form good coordinates we will
have to restrict $z$ to some  $|z | < \delta_0$ for some   $\delta_0 >0$.) 

In terms of  our  local trivializing coordinate $(x, z) \in W_i \times \R^c$,  the invariant subvariety $\mathcal P$ is given by $z = 0$
and the vector field $Y$ can be written as   
\begin{equation}
 \begin{split}
 x'&=f(x)+F(x,z),\\
 z'&=B(x)z+R(x,z),
 \end{split}
 \qquad
 F(x,z)=O(\|z\|),\quad R(x,z)=O(\|z\|^2)
\label{eq: tubular equations}
\end{equation}
where $f$ is the restriction of $Y$ to $\mathcal P$.  

For   $k\in K$, $f(k)=0$, and the matrix $B(k)$ is   the  linearization  $DY(k)$ on the normal fiber as described in the Proposition's hypothesis.  (Away from $K$, the matrix $B(x)$ has no intrinsic meaning, i.e. it is coordinate-dependent.)   Our hypothesis is that  $-B(k)$
is a Hurwitz matrix for each $k \in K$.   
We now follow the  classic prescription from linear control theory described above to create a fiber-wise Lyapunov function associated to
the linear flow $z' = -B(k) z$  in order to construct our claimed reverse Lyapunov  function $\mathcal E$. 

Let $B(x)^t$ denote  the adjoint of $B(x)$.
The classic prescription yields  the existence
of a unique positive-definite matrix $P(x)$  satisfying 
\begin{equation}\label{eq: normal Lyapunov identity}
 B(x)^t P(x)+P(x)B(x)=I 
\end{equation}
(Recall that $A = -B(x)$ is Hurwitz, so equation  \eqref{eq: normal Lyapunov identity} is equivalent to the  equation $A^t P +PA = -I$ discussed earlier
when we set  $P = P(x)$).  The solution $P(x)$  for this equation  persists and depends   smoothly on $x$
over any domain on which $-B(x)$ is smooth and   Hurwitz.   
(An explicit  formula for $P(x)$  is  
$P(x)=\int_0^\infty
       e^{-t B(x)^t}e^{-t B(x ) }\,dt$.) 
Now {\it if} we replace the $z'$ equation by $z' = B(x) z$, assuming that $x' = 0$  we would compute that
$\mathcal E' = \langle z, ( B(x)^t P(x) + P(x) B(x) ) z \rangle = \| z \|^2$.
Because our actual evolution is 
\eqref{eq: tubular equations}
 we get instead 
\begin{eqnarray}
\mathcal E' & = &  |z|^2 +  z^t DP(x)[x']z  + 2 z^T P(x)  O(|z|^2) \\
             & \ge  &  |z|^2 - \eps_1 C |z|^2   - 2 \eps_2 C |z|^2
\end{eqnarray}
where we will now describe the constants $\eps_1, C, \eps_2$  arising in the bounds of the second line.
 $Y$ vanishes on $K$ so  we can choose, for any $\eps_1 >0$
a  relatively compact  \nbhd $U_1$ of $K$ so small that $\| Y \| < \eps_1$ on $U_1$. 
On $U_1$ we find that  $|x'| < \|Y (x,z) \| < \eps_1$.  Since  $P(x)$ is smooth
we can find a  constant 
$C > 0$ such that  $|dP(x) \|, \|P(x) \|  \le C$ on any relatively compact \nbhd of $K$.  
Finally, given any $\eps_2 > 0$  restrict $|z| < \eps_2$ to define a \nbhd $U_2$ in our local trivialization coordinates.
Then $U = U_1 \cap U_2$ is  relatively compact and take $C$ the   associated constant.  We have that on  the open set $U$ of our coordinate \nbhd the
above bound holds.  
Finally,  given any  $\eps > 0$, choose $\eps_1, \eps_2 > 0$ so that 
$  C (\eps_1 + 2\eps_2) < \eps$ to arrive at  
$$ \mathcal E' \ge (1- \eps) |z|^2. $$
 
 If $P_0$ is any positive-definite $c \times c$ matrix then there exists a positive constant $\alpha_0$
 such that $\frac{\alpha_0}{2} |z|^2 > \langle z , P_0 z \rangle > \frac{2}{\alpha_0} |z|^2$.
 Our $P(x)$ is smooth function with values in the space of $c \times c$ positive-definite matrices
 and $K$ is compact. 
 It follows that there exists positive $\alpha$ and a cover of   $K$ with a finite number of local trivializations 
   $W_i \times \R^c$ with $|z| < \delta$ such that in each one we get 
$\frac{\alpha}{2} |z|^2 \ge \langle z, P(x) z \rangle  \ge  \frac{2}{\alpha} |z|^2$ and thus
$\frac{\alpha}{2} |z|^2 \ge \mathcal E (x, z)  \ge  \frac{2}{\alpha} |z|^2$ on an open \nbhd of $K$.  Combining with the last displayed equation
of the previous paragraph we get  
$$\mathcal E' > (1 - \eps) \frac{2}{\alpha} \mathcal E. $$
Choose $\eps$ such that $(1-\eps) > \frac{1}{2}$ to get the required inequality
\eqref{reverse Liap}.

Finally, we need to say some words about working in the different trivializing charts
and putting the whole mess together.  
Since the vector field $Y$ is globally defined and since our coefficients   arise by
viewing $Y$  in the charts, the matrices   $B(x)$ from the different charts fit together
and transform properly under  chart overlaps.  Since $P(x)$ is derived from $B(x)$ in the
same way in each chart, the $P(x)$'s also fit together properly under overlaps so that the function  $\mathcal E$ is
smooth and defined on the  union of charts $W_i \times D(\delta_i)$  covering $K$.  (Here $D_{\delta}$
is the disc $\{ z: |z| < \delta\} \subset \R^k$.) The constants $\alpha > 0$ may vary from chart to chart
but $K$ is compact so we only need a finite number of charts to cover $K$ and so  we can take as the global
$\alpha$ the maximum of the $\alpha$'s from the various charts.  
 
\end{proof}

 \begin{remark}  We thank the anonymous referee for suggesting something like Proposition \ref{prop: repulsion}.
 An earlier version of this paper used a version of the center-stable manifold theorem in
 place of that proposition. The usage and references for the  theorem were somewhat hazy in that version
 so that  the current proof is clearer than the proof in that  old version.  
 \end{remark}
  \section{Proof of Theorem \ref{thm;mini-coll}} \label{sec;mini-coll}

\begin{proof}

Given a initial configuration $q_0\ne 0$ and time $T>0$, by a standard argument of the direct method of calculus of variations, there exists a path $q: [0, T] \to \E_0$, which is a global minimizer of the following action functional  
$$ \A(\gm; 0, T) = \int_{0}^{T} \ey \langle \dot{\gm}, \dot{\gm} \rangle + U(\gm) \,dt $$ 
among all Sobolev paths starting from $q_0$ at $t=0$ and ending at a total collision at $t=T$. By Marchal's lemma (see \cite{C02}), $q(t)$ is collision-free, for all $t \in (0, T)$. Hence it is a total collision solution. 

For the rest of the theorem it is enough to prove $q(t)$ must be a spatial configuration, for all $t \in (0, T)$. By a contradiction argument let's assume there is a $t_0 \in (0, T)$, such that $q(t_0)$ is either a planar or collinear configuration. Then $V=span(q(t_0))$ is either a 1 or 2-dimensional linear subspace in $\R^3$.  

Let $g_V$ be an isometry of $\R^3$ with $V$ fixed under the action of $g_V$. We can define a new path $\tilde{q}$ with the same end points as below 
$$ \tilde{q}(t) = \begin{cases}
q(t), \; & \text{ when } t \in [0, t_0];\\
g_V(q(t))= (g_V(q_1(t)), \dots, g_V(q_4(t))), \; & \text{ when } t \in [t_0, T]. 
\end{cases}  
$$
Since $g_V$ is an isometry, $\A(\tilde{q}; 0, T) = \A(q; 0, T)$. 

Now notice that $span(\dot{q}(t_0)) \subset V$ can not hold, otherwise all masses must belong to $V$ at every instant. This means the new path $\tilde{q}$ can not be smooth at $t=t_0$, and by making a small enough local deformation near $\tilde{q}(t_0)$ we can get another path with strictly smaller action value. However this is a contradiction to the fact that $q$ is global minimizer. 
\end{proof}

 \section{Appendices}
 \subsection{A note on angular momentum.}  Every total collision solution has  angular momentum zero.
 For this reason it is traditional to impose the condition that the angular momentum is zero   when constructing  the collision manifold.
 We have not done so.  If we had done so we would be forced
 to work in the context of   singular varieties when struggling
 with the proof of Theorem \ref{thm;planar} since the zero-angular momentum subvariety of the spatial 4-body problem has   conical singularity along collinear  phase points.  
 
 Another  advantage of  allowing angular momentum when forming our   phase space is that
 it gives us a fuller understanding of the positive eigenspace of $L$ associated to  symmetry-induced kernel of the Hessian as described in Subsection \ref{ss: symmetry} and Corollary \ref{cor;invariant}. Recall that the eigenvalue $0$ for the Hessian gave rise to the pair of eigenvalues $0, -\nu_0/2$ for the linearized flow operator $L$.
 The `$0$' is easy to understand. Its associated eigenspace is the kernel of $L$ and corresponds to 
 the fact that the vector field is invariant under rotations. The positive   eigenvalue  $\frac{ -\nu_0}{2}$ is harder to understand and we will show that it
 is associated to the angular momentum $J$.
  
  We have $J(q, v) = \sum m_a q_a \wedge v_a$ which we  write symbolically as $J(q, v) = q \wedge v$.
Recall that $q = r s,  v = r^{-1/2} y$ and  $y = \nu s + w$ so that $J = r^{1/2} s \wedge y = r^{1/2} s \wedge w$.
 The bivector $\tilde J(s, w) = s \wedge w$ is the scale-invariant version of angular momentum and is a non-zero  analytic function
 which extends to our collision manifold.  From  $\tilde J = r^{-1/2} J$, $J' = 0$ 
 and $r' = r \nu$ we get the evolution equation $\tilde J' = - \frac{ \nu}{2} \tilde J$.  Take (some or all of) the components of 
 $\tilde J$ as forming   part  of a coordinate system on the collision manifold in a \nbhd of an equilibrium point.
 When we linearize at our equilibrium  point $\xi_0$ given by $(r, \nu, s, w) = (0, \nu_0, s_0, 0):= \xi_0$ with
 $\nu_0 < 0$  we get $ \delta \tilde J' = - \frac{ \nu_0}{2} \delta \tilde J$.  This accounts
 for the eigenvalue $-\nu_0/ 2$.  Symmetry-induced elements of the kernel of the Hessian
have  the form $As_0$ where $A \in \mathfrak g$ is a skew-symmetric operator acting
 simultaneously on each component $s_{0a}$ of $s_0$.  The associated
 eigenvector pair for $L$  is $e_1 = (0, 0, As_0, 0)$ with eigenvalue $0$  and $e_2 = (0, 0, As_0, \frac{ -\nu_0}{2} A s_0)$, 
 with  eigenvalue $- \nu_0/ 2$.   One verifies  that $d \tilde J (\xi_0) (e_2)= -\frac{\nu_0}{2}(s_0 \wedge A s_0 )\ne 0$
 so that  these $e_2$'s (as $A$ varies)  account for the $\delta \tilde J$-type eigenvalues with eigenvalue  $-\nu_0/ 2$.
  
\subsection{ A note on symmetries}

It has been conceptually helpful  at times to realize that 
\beq \mathfrak g(s_0) = (\mathfrak g (s_0) \cap Hor) \oplus (\mathfrak g (s_0) \cap V).
\Leq{gdirectsum}
The issue here   is that if   $S$ is a subspace of a direct sum
$A \oplus B$ then it is typically false that $S = (A \cap S) \oplus (B \cap S)$.
(Think of the diagonal $x=y$ viewed as a subspace of $\R^2 = \R \oplus \R$.)
But in our case it is true that $S = (A \cap S) \oplus (B \cap S)$

We  prove \eqref{gdirectsum} for the case of a planar configuration $s_0$  when $d=3$
then make a comment about how this proof generalizes to the full situation.
So we assume that $span(s_0) = \R^2$ is the xy plane with basis $e_1, e_2$
and perpendicular direction  spanned by $e_3$.   Identify
$\mathfrak g$ with $\R^3$ where $\omega \in \mathfrak g = \R^3$ acts on a configuration $s =  (s_1, s_2, \ldots s_N)$ by  $s \mapsto (\omega, s) \mapsto \omega \times s =
 (\omega \times s_1, \omega \times s_2, \ldots , \omega \times s_N)$.
Take $\mathfrak g$ and also split it into $\R^2 \oplus \R$ with the $\R$ factor spanned by $e_3$
and the $\R^2$ factor by $e_1, e_2$.    Now $\mathfrak g (s_0) = Span( e_1 \times s_0, e_2 \times s_0, e_3 \times s_0)$.
Cross product by $e_3$ maps the xy plane
to itself, and hence preserves horizontality.  It follows that $\R (e_3 \times s_0) =\mathfrak g (s_0) \cap Hor$.
(It is also true that $e_3 \times s_0 \ne 0$ since $e_3 \times s_0$ is the result of  rotating   $s_0$ by 90 degrees
within the horizontal plane.)    Now look at the action of $e_1$ on $s_0$.  Any component $s_{0j}$ of $s_0$ is of the form $a e_1 + be_2$ and $e_1 \times (a e_1 + b e_2) = b e_3$.  Thus $e_i \times s_0 \in V$, for $i=1, 2$. We have just shown that $Span(e_1 \times s_0, e_2 \times s_0) \subset V$ while
$Span(e_3 \times s_0) \subset Hor$.  Combining these last results yields \eqref{gdirectsum}.
  
For the general case of  \eqref{gdirectsum}, we have that $span(s_0) = \R^k \subset \R^d$ with $k < d$.  The vertical-horizontal splitting
induces a splitting of elements of  $\mathfrak g$ into two-by-two-block matrices of the form hor-hor,  hor-vert , vert-hor and vert-vert.
The hor-hor  component plays the role of $\omega = e_3$ in the above paragraph.  The vert-hor and hor-vert skew blocks
are spanned by bivectors of the form $e_h \wedge e_v$ with $e_h$ horizontal and $e_v$ vertical.  They play the role of
$\omega = e_1$ and $\omega = e_2$ above.    The vert-vert block acts trivially on $s_0$.  The proof then follows the same
basic lines.

\hfill\newline
\noindent{\bf Acknowledgement.} Both authors thank CIMAT for its hospitality where part of this work was done. 

\noindent{\bf Data Availability.} Data sharing is not applicable to this article as no datasets were generated or analyzed during the current study.

\noindent{\bf Conflict of interest} The authors declare no conflict of interest.

\bibliographystyle{abbrv}
\bibliography{RefFourBody}

\end{document}